\documentclass[12pt,a4paper]{article}
\usepackage{anysize}\marginsize{31mm}{31mm}{34mm}{33mm}
\usepackage[latin2]{inputenc}
\usepackage{t1enc}

\usepackage{fancyhdr}
\usepackage[english]{babel}
\usepackage{amssymb}
\usepackage{stmaryrd}
\usepackage{amsfonts}
\usepackage[namelimits,intlimits,sumlimits]{amsmath}
\usepackage{array}

\usepackage{amsthm}

\newtheorem{lemma}{Lemma}

\theoremstyle{definition}

\theoremstyle{plain}

\newtheorem*{theor}{Theorem}

\allowdisplaybreaks

\def\cite#1{[#1]}
\def\beq{\begin{equation}}
\def\eeq{\end{equation}}

\begin{document}

 \renewcommand{\headrulewidth}{0.5pt}
 
 \title{\bf  Empirical Processes and Schatte Model}
\author{Marko Raseta$^{\rm 1}$ }
\date{}
 \maketitle

\renewcommand{\thefootnote}{}



\footnote{$^{\rm 1)}$Department of Mathematics, University of York Email: {\tt
marko.raseta@york.ac.uk}}

Let $(X_n)_{n \in \mathbb N}$ be a sequence of independent, identically distributed random variables
and $S_k = X_1 + \dots + X_k$. Put $f_a(x) = I \{x \leq a\} - a$. Furthermore, let $F_n(t)= F_n(t, x)$ stand for
the empirical distribution function of the sample $\{S_1 x\},\dots, \{S_n x\}$.

\begin{theor}
Let $(X_n)_{n \in \mathbb N}$ be as above and suppose $X_1$ is bounded with bounded density.
Let
$$
\Gamma(s,t) = s(1- t) + \sum_{\varrho = 1}^\infty \mathbb E f_s(U) f_t(U + S_\varrho x) + \sum_{\varrho = 1}^\infty \mathbb E 
f_t(U) f_s(U + S_\varrho x),
$$
where $U$ is a $U(0, 1)$ variable independent of $(X_n)_{n \in \mathbb N}$.
Then for all fixed $x\ne 0$, and after suitably enlarging the probability space, there exist mean zero Gaussian processes 
$K_n(s)$ with covariance function $\Gamma_n(s,t)$ such that
\beq
\label{1}
\sup_{s,t \in [0,1]^2} \bigl|\Gamma_n(s,t) - \Gamma(s,t)\bigr|= O(n^{-\alpha}) \ \text{ for all } \alpha < 1/8
\eeq
and
\begin{gather}
\label{2}
\sup_{0 \leq s\le 1} \left|\sqrt{n} \left(F_n(s) - s\right) - K_n(s)\right| = O(n^{-\gamma}) \qquad
\text{a.s.\ for all } \gamma < 1/16.
\end{gather}

\end{theor}

\bigskip
Without loss of generality, it suffices to prove the theorem for $x=1$. We give some lemmas of which the first was proved in 
our
previous manuscript.

\begin{lemma}
\label{lem:1}
Let $f: \mathbb R \to \mathbb R$ denote a bounded measurable function satisfying
\begin{equation}\label{f}
f(x + 1) = f(x), \ \ \int\limits_0^1 f(x) dx = 0 \ \text{ and } \ \biggl(\int\limits_0^1 \bigl| f(x + h) - f(x) \bigr|^2 
dx\biggr)^{1/2} \leq Ch
\end{equation}
for some positive constant $C$.
Fix an integer $\ell \geq 1$ and define a sequence of sets by:

\goodbreak

\begin{align*}
I_1 &:= \{1,2,\dots, b\}\\
I_2 &:= \{p_1, p_1 + 1, \dots, p_1 + b_1\} \ \text{ where } p_1 \geq b + \ell + 2\\
\vdots &\\
I_n &:= \bigl\{p_{n - 1}, p_{n - 1} + 1, \dots, p_{n - 1} + b_{n - 1}\bigr\} \text{ where } p_{n - 1} \geq p_{n - 2} + b_{n - 
2} + \ell + 2.
\end{align*}
Then there exists a sequence $\delta_1, \delta_2, \dots$ of random variables, not depending on $f$, satisfying the following 
properties:
\begin{itemize}
\itemsep=0pt
\item[{\rm (i)}] $|\delta_n| \leq Ce^{-\lambda \ell}$ for all $n \in \mathbb N$ and some positive constants $C, 
    \lambda$.

\item[{\rm (ii)}] The random variables
\[
\sum_{i \in I_1} f(S_i), \sum_{i \in I_2} f(S_i - \delta_1), \dots, \sum_{i \in I_n} f(S_i - \delta_{n - 1})
\]
are independent.\\
\end{itemize}
\end{lemma}

Partition the interval $[0, n]$ into disjoint blocks,
\[
[0, n] = I_1 \cup J_1 \cup I_2 \cup J_2 \cup \dots \cup I_\ell \cup J_\ell
\]
where $|I_k| = \lfloor n^\alpha \rfloor$, $|J_k| = \lfloor n^\beta\rfloor$, $\ell \sim n^{1 - \alpha}$ for some positive 
reals
$\alpha, \beta$ to be chosen later.
Using Lemma~1 we can construct sequences $\bigl({\Delta_k}\bigr)_{k \in \mathbb N}$ and $\bigl({\Pi_k}\bigr)_{k \in \mathbb 
N}$ of random variables such that
\[
{\Delta_0} = {\Pi_0} = 0,
\]
\[
\bigl|{\Delta_k}\bigr| \leq C e^{-\lambda \lfloor n^\beta\rfloor}, \ \bigl|{\Pi_k}\bigr| \leq C e^{-\lambda \lfloor 
n^\alpha\rfloor}
\]
for $k = 1, \dots, \ell$ and
\begin{align*}
T_k^{(f_t)} &:= \sum_{j \in I_k} f_t \bigl(S_j x - {\Delta_{k - 1}}\bigr) \ \text{ and }\\
T_k^{*(f_t)} &:= \sum_{j \in J_k} f_t \bigl(S_j x - {\Pi_{k - 1}}\bigr)
\end{align*}
are sequences of independent random variables. Put
\begin{equation}\label{Af}
{A}^{(f)} := \|f\|^2 + 2 \displaystyle\sum\limits_{\varrho = 1}^\infty \mathbb E f(U) f(U + S_\varrho x)
\end{equation}
where $U$ is a uniform $(0,1)$ random variable, independent of $(X_j)_{j \geq 1}$.
Moreover, let
\[
m_k = k \left(\lfloor n^\alpha\rfloor + \lfloor n^\beta\rfloor\right) \ \text{ for } \ 1 \leq k \leq n.
\]

\begin{lemma}
\label{lem:2}
We have
\[
\sum_{k = 1}^\ell \mathbb V\text{\rm ar} \left({T_k}^{(f_t)}\right) \sim {A}^{(f_t)} n, \qquad
\sum_{k = 1}^\ell \mathbb V\text{\rm ar} \left({T_k}^{*(f_t)}\right) \sim {A}^{(f_t)} n^{1 - \beta}
\]
where $A^{(f_t)}$ is defined in (\ref{Af})  above.
\end{lemma}

\begin{proof}
To simplify the formulas,  we shall drop $f_t$ from the upper index because it is  constant throughout the argument.
We will establish the result for the long blocks, the corresponding result for short blocks can be obtained in an identical 
fashion.
We have
\[
\mathbb V\text{\rm ar} ({T_k}) = \sum_{j \in I_k}
 \mathbb E f^2 (S_j - \Delta_{k - 1}) + 2 \sum_{\rho = 1} ^{|I_k| - 1} \sum_{q = m_{k - 1} + 1} ^{m_{k - 1} + |I_k| - \rho} 
 \mathbb E f (S_q - \Delta_{k - 1}) (S_{q + \rho} - \Delta_{k - 1}) - L^{(k)}
\]
where
\[
L^{(k)} := \left(\sum_{j = m_{k - 1} + 1}^{m_{k - 1} + |I_k|} \mathbb E f\left(S_j - \Delta_{k - 1}\right)\right)^2.
\]
By the above assumptions and the bounds on $\Delta_k$ we have
\[
\left\| f(S_j - \Delta_{k - 1}) - f(S_j)\right\| \leq Ce^{-\lambda \lfloor n^\beta\rfloor}
\]
and
\[
\bigl|\mathbb E f(S_j)\bigr| \leq C e^{-\lambda j}.
\]
Thus
\[
L^{(k)} \leq C n^{2\alpha} e^{-\lambda\lfloor n^\beta\rfloor}.
\]
Now let
\[
\Lambda^{(k)} := \sum_{j = m_{k - 1} + 1}^{m_{k - 1} + |I_k|} \gamma_j, \qquad O^{(k)} := \sum_{j = m_{k - 1} + 1}^{m_{k - 1} 
+ |I_k|} \varepsilon_j
\]
where
\begin{align*}
\gamma_j &= \mathbb E f^2\bigl(S_j - \Delta_{k - 1}\bigr) - \mathbb E f^2(S_j)\\
\varepsilon_j &= \mathbb E f^2 (S_j) - \mathbb E f^2 \left(F_{\{S_j\}} (\{S_j\})\right).
\end{align*}
Repeating the argument above for $f^2 - \|f\|$ we get
\[
\sum_{j = m_{k - 1} + 1}^{m_{k - 1} + |I_k|} \mathbb E f^2 \bigl(S_j - \Delta_{k - 1}\bigr) = \|f\|^2\lfloor n^\alpha\rfloor 
+ \Lambda^{(k)} + O^{(k)}
\]
and
\begin{align*}
\left|\Lambda^{(k)}\right| &\leq C  n^\alpha e^{-\lambda \lfloor n^\beta\rfloor},\\
\left|O^{(k)}\right| &\leq C  n^\alpha e^{-\lambda \lfloor n^\beta\rfloor}.
\end{align*}
We now turn to the cross terms.
Define ${T_\varrho}^\ell = X_{q + 1} + \dots + X_{q + \varrho}$ and split the product expectation $\mathbb E f(S_q - 
\Delta_{k - 1}) f(S_{q + \varrho} - \Delta_{k - 1})$ into a sum of terms:
\begin{align*}
e_q &:= \mathbb E f\left(S_q - \Delta_{k - 1}\right) f \left(S_{q + \rho} - \Delta_{k - 1}\right) - \mathbb E f(S_q) 
f\left(S_{q + \rho} - \Delta_{k - 1}\right),\\
g_q &:= \mathbb E f\left(S_q\right) f \left(S_{q + \rho} - \Delta_{k - 1}\right) - \mathbb E f(S_q) f\left(S_{q + \rho} 
\right),\\
h_q &:= \mathbb E f\left(S_q\right) f \left(S_{q + \rho}\right) - \mathbb E f\left(F_{\{S_q\}}(\{S_q\})\right) f\left(S_{q + 
\rho}\right),\\
i_q &:= \mathbb E f\left(F_{\{S_q\}}(\{S_q\})\right) f\left(S_q + {T_\rho}^q\right) - \mathbb E f\left(F_{\{S_q\}} 
(\{S_q\})\right) f\left(F_{\{S_q\}}(\{S_q\}) + {T_\rho}^q\right),\\
{c_\rho}^q &:= \mathbb E f \left(F_{\{S_q\}} (\{S_q\})\right) f \left(F_{\{S_q\}}(\{S_q\}) + {T_\rho}^q\right).
\end{align*}
Here $F_{\{S_q\}}(\{S_q\})$ is a uniformly distributed random variable independent of ${T_\varrho}^\ell$ and thus letting $U$ 
denote a uniform $(0,1)$ random variable independent of $(X_j)_{j \geq 1}$ we have
\[
{c_\rho}^q = c_\rho = \mathbb E f(U) f (U + S_\rho)
\]
is a $q$-independent quantity.
Exactly as before
\begin{align*}
|e_q| \leq C e^{-\lambda\lfloor n^\beta\rfloor}, \ |g_q| \leq C e^{-\lambda\lfloor n^\beta\rfloor}, \
|h_q| \leq C e^{-\lambda q}, \ |i_q| &\leq C e^{-\lambda q}.
\end{align*}
Thus letting
\begin{alignat*}{2}
E^{(k)} &:= 2 \sum_{\rho = 1}^{|I_k| - 1} \sum_{q = m_{k - 1} + 1}^{m_{k - 1} + |I_k| - \rho} e_q, \quad &
G^{(k)} &:= 2 \sum_{\rho = 1}^{|I_k| - 1} \sum_{q = m_{k - 1} + 1}^{m_{k - 1} + |I_k| - \rho} g_q,\\
H^{(k)} &:= 2 \sum_{\rho = 1}^{|I_k| - 1} \sum_{q = m_{k - 1} + 1}^{m_{k - 1} + |I_k| - \rho} h_q, \quad &
I^{(k)} &:= 2 \sum_{\rho = 1}^{|I_k| - 1} \sum_{q = m_{k - 1} + 1}^{m_{k - 1} + |I_k| - \rho} i_q
\end{alignat*}
we have
\begin{align*}
|E^{(k)}| &\leq Cn^{2\alpha} e^{-\lambda \lfloor n^\beta\rfloor}, \quad
|G^{(k)}| \leq Cn^{2\alpha} e^{-\lambda \lfloor n^\beta\rfloor},\\
|H^{(k)}| &\leq Cn^{2\alpha} e^{-\lambda \lfloor n^\beta\rfloor}, \quad
|I^{(k)}| \leq Cn^{2\alpha} e^{-\lambda \lfloor n^\beta\rfloor}.
\end{align*}
Furthermore,
\begin{align*}
\sum_{\rho = 1}^{|I_k| - 1} \sum_{q = m_{k - 1} + 1}^{m_{k - 1} + |I_k| - \rho} {c_\rho}^q
&= \sum_{\rho = 1}^{|I_k| - 1} \sum_{q = m_{k - 1} + 1}^{m_{k - 1} + |I_k| - \rho} {c_\rho} =\\
&= |I_k| \sum_{\rho = 1}^\infty c_\rho - |I_k| \sum_{\rho = |I_k|}^\infty c_\rho - \sum_{\rho = 1}^{|I_k| - 1} \rho c_\rho.
\end{align*}
Putting all this together one can see that
\[
\sum_{k = 1}^\ell \mathbb V\text{\rm ar}\left({T_k}^{(f_t)}\right) \sim n {A}^{(f_t)} . \qedhere
\]
\end{proof}

\begin{lemma}
\label{lem:3}
$A^{(f_{t- s})} \leq B(t - s) \log \left((t - s)^{-1}\right)$
for some suitably chosen constant $B$.
\end{lemma}

\begin{proof}
Clearly
\[
A^{(f_{t-s})} = \left((t - s) - (t - s)^2\right) + 2 \sum_{\rho = 1}^\infty \mathbb E f_{(t - s)}(U) f_{t - s} (U + S_\rho 
x).
\]
Define a sequence of uniformly distributed random variables $U_\varrho$ by
\[
U_\rho := \{U + S_\rho x\}.
\]
By the definition of $f$
\begin{align*}
\left|\mathbb E f_{(t - s)}(U)f_{t - s} (U + S_\rho x)\right|
&= \left|\mathbb C\text{\rm ov}\left(I \{ U \in (s, t]\}, I \{U_\rho \in (s, t]\}\right)\right| \\
&\leq t - s
\end{align*}
by simple algebra. On the other hand, the machinery we used in the proof of Lemma~\ref{lem:2} yields the estimate
\[
\left|\mathbb E f_{t - s}(U) f_{t - s} (U + S_\rho x)\right| \leq C e^{-\lambda \rho}.
\]
Then:
\begin{align*}
A^{(f_{t - s})} &\leq (t - s) + 2 \sum_{\rho = 1}^\infty \min \left(t - s, C e^{-\lambda \rho}\right) \\
&\leq (t - s) + 2 \sum_{\rho = 1}^\omega (t - s) + 2 \sum_{\rho > \omega} C e^{-\lambda \rho}
\end{align*}
and the result follows by choosing $\omega= \log (t - s)^{-1}$ (whenever  $\log (t - s)^{-1} > 1$).
\end{proof}

We now turn to the proof of our theorem. Using Borel--Cantelli lemma, one can see that statement (\ref{2}) of our Theorem is 
equivalent to
\begin{equation}
\label{2a}
\sup_{s \in [0,1]} \left|n^{-1/2} \sum_{k = 1}^\ell {T_k}^{(f_s)} + n^{-1/2} \sum_{k = 1}^\ell T_k^{*(f_s)} - K_n(s) \right|
= O(n^{-\gamma}) \qquad \text{a.s.}
\end{equation}
Let $\delta_n = n^{-\varepsilon}$ for some $\varepsilon>0$ to be specified later,  and define points $z_j = j \cdot 
\delta_n$, $0 \leq j \leq r_n$, where $r_n$ is the largest integer with $r_n \delta_n \leq 1$.
Define sequences of random vectors by
\begin{alignat*}{2}
&{\text{\rm (a)}}\quad & \mathbf{X}_j &:= \left(I \left\{\left\{S_j x - {\Delta_{k - 1}}\right\} \leq z_1\right\}, \dots,
I \left\{\left\{S_j x - {\Delta_{k - 1}}\right\} \leq z_{n}\right\}\right),\\
&{\text{\rm (b)}}\quad & \mathbf{Y}_j &:= \left(I \left\{\left\{S_j x - {\Pi_{k - 1}}\right\} \leq z_1\right\}, \dots,
I \left\{\left\{S_j x - {\Pi_{k - 1}}\right\} \leq z_{n}\right\}\right),\\
&{\text{\rm (c)}}\quad & \boldsymbol{\xi}_k &:= \sum_{j \in I_k} \left(\mathbf{X}_j - \mathbb E \mathbf{X}_j\right),\\
&{\text{\rm (d)}}\quad & \boldsymbol{\eta}_k &:= \sum_{j \in J_k} \left(\mathbf{Y}_j - \mathbb E \mathbf{Y}_j\right).
\end{alignat*}
We apply Lemma~4 of Berthet and Mason(see (\cite{1})  to $(\boldsymbol{\xi}_k)_{k \in \mathbb N}$.
It is evident from the construction that, after enlarging the probability space, there exists a sequence 
$({\boldsymbol{\xi}_k}^*)_{k \in \mathbb N}$ of independent, zero-mean Gaussian random vectors with the same covariance 
structure as the sequence $(\boldsymbol{\xi}_k)_{k \in \mathbb N}$ such that
\[
\mathbb P \left\{\left| \sum_{j = 1}^\ell (\boldsymbol{\xi}_j - {\boldsymbol{\xi}_j}^*)\right| > v \right\}
\leq C_1 {r_n}^2 \exp \left(-C_2 v \big/ n^{(\frac{5\varepsilon}{2} + \alpha)}\right),
\]
where $r_n$, $\varepsilon$ and $\alpha$ are as described previously and $C_1$ and $C_2$ are constants.
For $s\in (z_{p - 1}, z_p]$, $1 \leq p \leq r_n$ we define
\[
K_n(s) = n^{-1/2} \sum_{k = 1}^\ell {\boldsymbol{\xi}_k}^*(p) \ \text{ and } \ {K_n}^*(s) = n^{-1/2} \sum_{k = 1}^\ell 
{\eta_k}^*(p),
\]
where $(\boldsymbol{\eta}_k^*)_{k \in \mathbb N}$ are the exact analogue of $(\boldsymbol{\xi}_k^*)_{k \in \mathbb N}$ for 
the short blocks.
By construction, these are Gaussian processes whose covariances we denote by $\Gamma_n(s,t)$ and ${\Gamma_n}^*(s,t)$, 
respectively.
Assume that $1/2 - \gamma > 5\varepsilon/2 + \alpha$, $\alpha > \beta$.
Then the Borel--Cantelli lemma gives
\begin{align}
\label{eq:4}
&\max_{s \in H_n} \left| n^{-1/2} \sum_{k = 1}^\ell {T_k}^{(f_s)} + n^{-1/2} \sum_{k = 1}^\ell {T_k}^{* (f_s)} - K_n(s) - 
n^{\frac{\beta - \alpha}{2}} {K_n}^*(s)\right| = O(n^{-\gamma}) \quad \text{a.s.} \nonumber
\end{align}
where  $H_n = \left\{z_0, \dots z_{r_n}\right\}$. Another Borel--Cantelli argument gives, under the additional assumption 
that $\alpha - \beta > 2\gamma$,
\[
n^{\frac{\beta - \alpha}{2}} {K_n}^*(s) = o(n^{-\gamma}),\quad \text{a.s.}
\]
Thus, under the conditions
\begin{equation}\label{cond}
1/2- \gamma > 5\varepsilon/2 + \alpha, \quad \alpha > \beta, \quad \alpha - \beta > 2\gamma
\end{equation}
we have
\begin{gather}
\label{Hn}
\max_{s\in H_n} \left| n^{-\frac12} \sum_{k = 1}^\ell {T_k}^{(f_s)} + n^{-\frac12} \sum_{k = 1}^\ell {T_k}^{*(f_s)} - 
K_n(s)\right| = O(n^{-\gamma}) \qquad \text{a.s.} \\
\nonumber
\end{gather}
Borel--Cantelli lemma gives
Next we show now that \eqref{Hn} holds with $\max_{s\in H_n}$ replaced by $\sup_{s\in [0, 1]}$.
To this end, we have to control the fluctuations of the corresponding large and small block processes between points of 
$H_n$.
By Bernstein's inequality we have, for $s \in (z_p, z_{p + 1}]$
\begin{align*}
&\mathbb P\left(\left|\sum_{k = 1}^\ell {T_k}^{(f_{s - z_p})}\right| > n^{\frac12 - \gamma}\right) \\
&\leq 2 \exp \left( - \frac{\frac12 n^{1 - 2\gamma}}{\sum\limits_{k = 1}^\ell \mathbb V\text{\rm ar } {T_k}^{(f_{s - z_p})} + 
2n^\alpha n^{\frac12 - \gamma}}\right).
\end{align*}
Using Lemmas~\ref{lem:2} and \ref{lem:3} we deduce that the fluctuations are insignificant if $1/2 > \alpha + \gamma$ and 
$\gamma < \varepsilon/2$.
Since $\alpha > \beta$ there will be no further conditions for controlling the smaller block fluctuations.

It remains to impose conditions implying the  uniform convergence of the $\Gamma_n(s,t)$ and ${\Gamma_n}^*(s,t)$ to 
$\Gamma(s,t)$, over all
$s, t\in [0, 1]$. To this end, using the same ideas as before and the fact that
\begin{align*}
\mathbb C\text{\rm ov}\left(K_n(s), K_n(t)\right) &= \mathbb C\text{\rm ov}\left(K_n(z_{i + 1}), K_n(z_{j + 1})\right)\\
&\quad\text{for all } \ s \in (z_i, z_{i + 1}], \ t \in (z_j, z_{j + 1}]
\end{align*}
we get
\[
\left|\Gamma_n (z_i, z_j) - \Gamma(z_i, z_j)\right| \leq C/n^\alpha.
\]
Moreover, using Lemma~\ref{lem:3} we immediately deduce that
\[
\sup_{s,t \in [0,1]^2} \left|\Gamma_n(s,t) - \Gamma(s,t)\right| \leq C (\log n) / n^\varepsilon
\]
and that, under the condition $\varepsilon > \alpha$, the second statement (2) of the Theorem holds as well.

Finally, we will find the largest value of $\gamma$ for which our argument works.
Clearly,  we  have to maximize $\gamma$ subject to
$1/2 - \gamma > 5\varepsilon/2 - \alpha$, $\alpha - \beta > 2\gamma$, $\gamma < \varepsilon/2$,
$\varepsilon > \alpha$, $\alpha > \beta$, $\alpha > 0$, $\beta > 0$, $\gamma > 0$, $\varepsilon > 0$.
Then $\gamma < 1/2 - \alpha - 5\varepsilon/2 < 1/2 (1 - 7\alpha)$.
Thus $\gamma < \min (\alpha/2, 1/2 (1 - 7\alpha))$, the other constraints are not important.
By plotting the corresponding feasible region we see that the largest possible value of $\gamma$ corresponds
to the intersection of the lines $J = \alpha/2$, $J = 1/2 (1 - 7\alpha)$, whence  $\alpha = 1/8$.
Thus $\gamma < 1/2 \cdot 1/8 = 1/16$ and the proof is complete.

\end{document}